\documentclass[a4paper,12pt]{amsart}  
\usepackage{a4wide,amsmath,amssymb,amsthm}     
\usepackage[utf8]{inputenc}
\usepackage{enumitem}
\usepackage{color}

\theoremstyle{plain}      
 
\newtheorem{thm}{Theorem}[section]   
\newtheorem{question}[thm]{Question}   
\newtheorem{theorem}[thm]{Theorem}     
     
\newtheorem{corollary}[thm]{Corollary}     
     
\newtheorem{lemma}[thm]{Lemma}     
\newtheorem{prop}[thm]{Proposition}     
\newtheorem{proposition}[thm]{Proposition}     
\newtheorem{conjecture}[thm]{Conjecture}     

\newtheorem{hypothesis}[thm]{Hypothesis}  
\theoremstyle{definition}      
\newtheorem{defn}[thm]{Definition}     
\newtheorem{definition}[thm]{Definition}  
\newtheorem{example}[thm]{Example} 
\newtheorem{remark}[thm]{Remark}

\numberwithin{equation}{section}  

\DeclareMathAlphabet{\doba}{U}{msb}{m}{n}

\gdef\mR{\doba{R}}

\gdef\mZ{\doba{Z}}         

\def\J{P}

\def\Sym{\mathrm{Sym}}

\def\Id{\mathrm{Id}}

\def\T{\mathrm{T}}

\def\Ric{{\mathrm{Ric}}}

\def\di{{\rm d}}
\def\tr{\mathrm{tr}}
\def\e{{\varepsilon}}
\def\vol{{\mathrm{vol}}}

\def\rk{\mathrm{rk}}

\let\<\langle 
\let\>\rangle

\newcommand{\definedas}{\mathrel{\raise.095ex\hbox{\rm :}\mkern-5.2mu=}}

\author{Andrei Moroianu, Mihaela Pilca}

\address{Andrei Moroianu \\ Université Paris-Saclay, CNRS,  Laboratoire de mathématiques d'Orsay, 91405, Orsay, France, 
and Institute of Mathematics “Simion Stoilow” of the Romanian Academy, 21 Calea Grivitei, 010702 Bucharest, Romania}
\email{andrei.moroianu@math.cnrs.fr}

\address{Mihaela Pilca\\Fakult\"at f\"ur Mathematik\\
Universit\"at Regensburg\\Universit\"atsstr. 31 
D-93040 Regensburg, Germany,
and Institute of Mathematics “Simion Stoilow” of the Romanian Academy, 21 Calea Grivitei, 010702 Bucharest, Romania}
\email{mihaela.pilca@mathematik.uni-regensburg.de}

\begin{document}   



	\title{Reducible Riemannian manifolds with conformal product structures}

	\begin{abstract} We study conformal product structures on compact reducible Riemannian manifolds, and show that under a suitable technical assumption, the underlying Riemannian mani\-folds are either conformally flat, or triple products, \emph{i.e.} locally isometric to  Riemannian manifolds of the form $(M,g)$ with $M=M_1\times  M_2\times  M_3$ and $g=e^{2f}g_1+g_2+g_3$, where $g_i$ is a Riemannian metric on $M_i$, for $i\in\{1,2,3\}$, and $f\in C^\infty(M_1\times M_2)$.
	\end{abstract}
	
	\maketitle

	\section{Introduction}
	
	Conformal product structures were formally introduced in \cite{bm2011}, but they appear in the literature several decades earlier. They can be defined in various ways: either in terms of Weyl connections (see for instance \cite{g1995},  \cite{c2000}), or by means of conformal foliations (\emph{e.g.} \cite{v1979}), or finally by locally describing the metrics in the conformal class (\emph{cf.} \cite{bm2011}). 
	
	The point of view that we adopt here is the following. If $(M,c)$ is a conformal manifold, a conformal product structure on $(M,c)$ is defined as a Weyl connection $D$ (i.e. a torsion-free linear connection preserving the conformal class) with {\em reducible} holonomy representation. 
	
	Recall that by fixing a background metric $g$ in the conformal class $c$, every Weyl connection $D$ determines a 1-form $\theta$ (called the Lee form with respect to $g$), by the formula $Dg=-2\theta\otimes g$, and conversely, a Riemannian $g\in c$ together with any 1-form $\theta$ determines a unique Weyl connection $D$ whose Lee form with respect to $g$ is $\theta$ (see Definition \ref{weyl} and Equation \eqref{ko} below). A Weyl connection is called closed or exact if its Lee form with respect to some metric in $c$ is closed or exact. This definition does not depend on the choice of the metric in the conformal class and it is easy to see that a Weyl connection is exact (or closed) if and only if it is globally (or locally) the Levi-Civita connection of a metric in $c$.
	
	It follows from Merkulov-Schwachhöfer's classification of possible  irreducible holonomies of torsion-free linear connections \cite{ms1999} that every non-closed Weyl connection with irreducible holonomy representation on a conformal $n$-dimensional manifold has holonomy equal to the full conformal group $\mathrm{CO}(n)$ if $n\ne 4$. This result already constitutes a strong motivation for investigating Weyl connections with reducible holonomy.
	
	The tangent bundle of a manifold $M$ endowed with a conformal product structure $D$ splits by definition as the orthogonal direct sum of two $D$-parallel distributions $\T M=T_1\oplus T_2$, which are both integrable due to the fact that $D$ is torsion-free. 
	
	The problem we tackle in this paper is part of a more general program, which aims to classifying conformal product structures $D$ on compact conformal manifolds $(M,c)$ which carry a metric $g\in c$ with special holonomy. In order to avoid trivial solutions, we always assume that $D$ is different from the Levi-Civita connection of $g$. 
	
	According to the Berger-Simons holonomy theorem, if the metric $g$ has special holonomy, then $(M,g)$ is either Kähler, or Einstein, or reducible. The first two cases have already been studied recently: the problem was completely solved when $g$ is Kähler \cite{mp2025}, and a partial solution, under some simplifying assumption, was obtained when $g$ is Einstein \cite{mp2024b}, \cite{mp2025a}. Previous results on the related problem of conformally Einstein product spaces have been obtained by W. Kühnel und H.-B. Rademacher  \cite{kr1997},  \cite{kr2009},  \cite{kr2016}.  We study here the last remaining case, when $g$ has reducible holonomy, more precisely by asking the following:
	\begin{question} \label{q} Characterize compact Riemannian manifolds $(M,g)$ with reducible holonomy, which carry a conformal product structure $D$,  other than the Levi-Civita connection of $g$.
	\end{question}
	
	Note that the above framework can be characterized by means of tensor fields on $(M,g)$ in the following way. The connection $D$ is determined by the metric $g$ and a 1-form $\theta$, and the $D$-parallel decomposition of $\T M=T_1\oplus T_2$ is determined (up to the choice of the two factors) by an orthogonal $D$-parallel involution $S$ whose $\pm1$ eigenspaces are exactly $T_1$ and $T_2$. Similarly, the reducibility of $(M,g)$ is equivalent to the existence of a $\nabla$-parallel orthogonal involution $\J$ different from $\pm\Id_{\T M}$, where $\nabla$ is the Levi-Civita connection of $g$. Question \ref{q} is thus equivalent to the existence of a triple $(\theta,S,\J)$ on $(M,g)$, where $\theta$ is a non identically vanishing 1-form, and $S,\J$ are orthogonal involutions different from $\pm\Id_{\T M}$ satisfying 
	\begin{equation}DS=0,\qquad \nabla\J=0,
	\end{equation}
	where $D$ is determined by $\nabla$ and $\theta$ by means of Equation \eqref{ko} below.
	
	It turns out that solutions to Question \ref{q} can be constructed on products of three manifolds. More precisely, if $g_i$ is a Riemannian metric on a compact manifold $M_i$ for $i\in\{1,2,3\}$ and $f$ is a smooth function on $M_1\times M_2$, then the metric $g:=e^{2f}g_1+g_2+g_3$ on $M:=M_1\times M_2\times M_3$ is obviously a Riemannian product metric between $(M_1\times M_2, e^{2f}g_1+g_2)$ and $(M_3,g_3)$. Moreover, the conformal class of $g$ also admits a conformal product structure $D$, different from the Levi-Civita connection of $g$ if the warping function $f$ is non-constant along $M_2$, cf. Example \ref{etp} below. Riemannian manifolds which are locally isometric to the ones constructed in this way are called {\em triple products}.
	
	Another class of solutions is given by compact conformally flat manifolds. Indeed, as noticed at the beginning of Section 4 in \cite{bfm2023}, if $D$ is a flat Weyl connection (which is in particular closed) on a compact conformal $n$-dimensional manifold $(M,c)$ with $n\ge 3$, then by a result of Fried  \cite{Frie}, the universal cover $\widetilde M$, endowed with the Riemannian metric whose Levi-Civita connection is the pull-back of $D$ to $\widetilde M$, is isometric to $\mR^n\setminus \{0\}$ and $\pi_1(M)$ is a semi-direct product $K \rtimes \mZ$ between a finite group of isometries of $\mR^n$ fixing the origin and a group generated by a homothety of ratio $\lambda < 1$. In polar coordinates, the flat space $\mR^n \setminus {0}$ can be seen as the product $\mR_+ \times S^{n-1}$ together with the cone metric $dr^2 + r^2 g_S$ where $g_S$ is the round metric on the sphere. Consequently, $\pi_1(M)$ acts by isometries with respect to the metric $\frac{1}{r^2} dr^2 + g_S$, which is conformal to the previous metric and descends to $M$. In addition, this metric is a product metric, so it is reducible, implying that all compact manifolds carrying a flat Weyl connection are solutions of Question \ref{q}. 
	
	Our purpose here is to collect evidence in favor of the following:
	\begin{conjecture} \label{c} A compact connected Riemannian manifold $(M,g)$ with reducible holonomy which carries a conformal product structure $D$ different from the Levi-Civita connection of $g$ is either conformally flat, or a triple product.
	\end{conjecture}
	
	Note first that by \cite[Thm. 6.3]{m2019}, Conjecture \ref{c} is true when $D$ is exact, i.e. when $D$ is the Levi-Civita connection of some metric on $M$ which is globally conformal to $g$. The conjecture was also proved when $D$ is closed but non-exact. Indeed, conformal product structures with closed non-exact and non-flat Weyl connection are exactly the so called locally conformally product (LCP) structures \cite{f2024}.  The study of LCP manifolds started with a construction given by V. Matveev and Y. Nikolayevsky \cite{mn2015} as a counterexample to a conjecture by Belgun and the first author \cite{bm2016}. A structure theorem was first obtained in the analytic case by V. Matveev and Y. Nikolayevsky \cite{mn2017} and afterwards extended to the smooth case by Kourganoff \cite{k2019} (see also \cite{fz2025} for another proof). In the LCP setting, Conjecture \ref{c} follows from \cite[Thm. 4.7 and Rem. 4.11]{bfm2023}. 
	
	In this paper we make further progress towards the proof of the conjecture in the remaining case, where $D$ is non-closed, by showing that it holds true under any of the following extra assumptions:
	\begin{enumerate}[label=(\roman*)]
			\item $S\theta=\varepsilon\theta$, for some $\varepsilon\in\{ -1,1\}$; this is equivalent to saying that $\theta$ vanishes on vectors from one of the two distributions of the conformal product.
			\item $\delta S\theta=\varepsilon\delta\theta$, for some $\varepsilon\in\{ -1,1\}$; this amounts to saying that the restriction of $\theta$ to one of the two distributions of the conformal product is co-closed.
			\item $S\J=\J S$; this commutation is equivalent to the fact that the tangent bundle of $M$ decomposes in an orthogonal direct sum of distributions $\T M=T_{1,1}\oplus T_{1,2}\oplus T_{2,1}\oplus T_{2,2}$ such that $T_{1,1}\oplus T_{1,2}$ and $T_{2,1}\oplus T_{2,2}$ are $D$-parallel and $T_{1,1}\oplus T_{2,1}$ and $T_{1,2}\oplus T_{2,2}$ are $\nabla$-parallel. 
			\item $\tr (S\J)$ is constant.
			\item $\delta(S\theta)=0$.
			\item $\tr(\J)=0$; this is equivalent to the fact that $\T M$ has a $\nabla$-parallel splitting in an orthogonal direct sum of two distributions of the same dimension.
		\end{enumerate}	
		Note that conditions (i)-(iv) are automatically satisfied on a triple product. Moreover, one has the obvious implications (i)$\Rightarrow$(ii) and (iii)$\Rightarrow$(iv).
	
	Our main result is the following:
	\begin{thm} \label{mainthm}
		Let $(M,g)$ be an oriented, compact connected Riemannian manifold of dimension $n\ge 3$ endowed with a Weyl connection~$D$, whose Lee form $\theta$ with respect to $g$ is not identically zero. We assume that both the Levi-Civita connection of $g$ and the Weyl connection $D$ have reducible holonomy, and that at least one of the conditions (i)-(vi) above is satisfied. 
		Then $(M,g)$ is a triple product.		
	\end{thm}
	
	The proof is based on an integral formula, somewhat similar to the one used in the proof of \cite[Thm. 4.1]{mp2025}. The main idea is to exploit the fact that $S$ is $D$-parallel in order to obtain a formula for the Riemannian curvature tensor $R$ acting on $S$, then to use the symmetries of $R$ given by the holonomy reduction, together with some suitable traces, in order to obtain a scalar identity involving a divergence term, and some further terms whose sign can be controlled. 
	
	However, unlike the proof of \cite[Thm. 4.1]{mp2025}, where the Kähler symmetries of $R$ together with the Cauchy-Schwarz inequality give a definite sign in the corresponding scalar identity, the fact that the holonomy reduction in the present case is determined by an orthogonal involution instead of an orthogonal complex structure, gives some extra terms in the scalar identity whose sign can not be controlled, unless we make one of the assumptions (i)-(vi) above.
	
	The structure of the paper is the following. After some preliminaries on Weyl connections and conformal product structures, we introduce triple products in Section 3, and give their tensorial characterization in Proposition \ref{proptriple}. In Section 4 we obtain the scalar identity mentioned above (Proposition \ref{propscalarformula}), and in Section 5 we prove that the coefficients involved in it are non-negative provided that one of the assumptions (i)-(vi) holds. Finally, in Section 6 we prove our main theorem, by distinguishing the cases where one of the $D$-parallel distributions has rank 1 or both have rank at least 2, and in the Appendix we prove several divergence formulas used in the proof of Proposition \ref{propscalarformula}.
	
	{\bf Acknowledgment.} This work was partly supported by the PNRR-III-C9-2023-I8 grant CF 149/31.07.2023 {\em Conformal Aspects of Geometry and Dynamics}.
	
	\section{Definitions and notation}

Let $(M,g)$ be a Riemannian manifold of dimension $n\ge 2$. We denote by $\nabla$ the Levi-Civita connection of $g$ and by $\sharp:\T^*M\to \T M$ and $\flat:\T M\to \T^* M$ the musical isomorphisms defined by $g$, which are $\nabla$-parallel and inverse to each other. In order to simplify notation, we will sometimes simply write $\langle\cdot,\cdot\rangle$ instead of the metric $g$, and denote the associated norm by $\|\cdot\|$. Furthermore we will identify vector fields with their dual $1$-forms with respect to $g$, when there is no risk of confusion.

We introduce the following notation: If $F$ is a non-trivial orthogonal  involution of $\T M$, then $F$ has exactly two eigenvalues, $-1$ and $1$, and we denote by $E_{+}(F)$ the eigenspace of  $F$ corresponding to the eigenvalue $1$ and by $E_{-}(F)$ the eigenspace of  $F$ corresponding to the eigenvalue $-1$. By a slight abuse of notation we denote these eigenspaces as $E_{\varepsilon}(F)$, for $\varepsilon\in\{-1,1\}$.  On a Riemannian manifold $(M,g)$, an endomorphism  $F$  of $\T M$ is called symmetric if $F$ is self-adjoint with respect to the metric $g$, \emph{i.e.} $g(FX, Y)=g(X, FY)$, for all vector fields $X,Y$. Clearly, an involution is symmetric if and only if it is an orthogonal endomorphism of $(\T M, g)$.

\begin{remark}\label{remgred}
	If $(M,g)$ is a Riemannian manifold with Levi-Civita connection $\nabla$, then the following assertions are equivalent:
	\begin{enumerate}[label=(\roman*)]
		\item The metric $g$ has reducible holonomy.
		\item There is an orthogonal non-trivial $\nabla$-parallel decomposition of the tangent bundle.
		\item There exists a $\nabla$-parallel $g$-orthogonal involution $\J\in\mathrm{End}(\T M)$ different from $\pm \Id_{\T M}$.
	\end{enumerate}
For the equivalence between $(ii)$ and $(iii)$ it suffices to notice that every orthogonal splitting \mbox{$\T M=D_1\oplus D_2$} is equivalent to an orthogonal involution $\J$ by defining $E_+(\J):=D_1$ and $E_-(\J):=D_2$. Moreover, $D_1$ and $D_2$ are $\nabla$-parallel if and only if $\J$ is $\nabla$-parallel.
\end{remark}

	\begin{defn}\label{weyl} A {\em Weyl connection} on $(M,g)$ is a torsion-free linear connection $D$  satisfying $Dg=-2\theta\otimes g$ for some $1$-form $\theta\in\Omega^1(M)$, called the {\em Lee form} of $D$ with respect to $g$. 
	\end{defn}
	
   The conformal Koszul formula  \cite{g1995} shows that $D$ is completely determined by its Lee form with respect to $g$:
	\begin{equation}\label{ko} D_XY=\nabla_XY+\theta(Y)X+\theta(X)Y-\langle X, Y\rangle\theta^\sharp,\qquad\forall X,Y\in \Gamma(\T M).
	\end{equation}
	
	\begin{defn}\label{cp} 
		A {\em conformal product structure} on $(M,g)$ is a Weyl connection $D$ together with a decomposition of the tangent bundle of $M$ as  \mbox{$\T M=T_1\oplus T_2$}, where $T_1$ and $T_2$ are orthogonal $D$-parallel non-trivial distributions. The {\em rank} of a conformal product structure is defined to be the smallest of the ranks of the two distributions $T_1$ and $T_2$. A conformal product structure $D$ on $(M,g)$ is {\em orientable} if the $D$-parallel distributions $T_1$ and $T_2$ are orientable.
	\end{defn}

Let us recall the following tensorial characterization of conformal product structures (\emph{cf.} \cite[Lemma~2.4]{mp2025}): 

\begin{remark}\label{remcps}
	On a Riemannian manifold $(M, g)$  with Levi-Civita connection $\nabla$ the following assertions are equivalent:
	\vspace{-0.2cm}
	\begin{enumerate}[label=(\roman*)]
		\item There exists a conformal product structure on $(M,g)$.
		\item There exists a  $g$-orthogonal involution $S$ of $\T M$ different from $\pm \Id_{\T M}$ and a $1$-form $\theta$ on $M$, such that 
		\begin{equation}\label{derivS}
			\nabla_X S=SX\odot \theta^\sharp - S\theta^\sharp\odot X, \quad \forall X\in\Gamma(\T M),
		\end{equation}
		\end{enumerate}	
where for any two vectors $X, Y\in\T M$ the symmetric endomorphism $X\odot Y$ is defined by
$$(X\odot Y)(Z):=\langle X,Z\rangle Y+\langle Y, Z\rangle X,\qquad\forall Z\in\T M.$$
More precisely, the orthogonal involution $S$ of $\T M$ is defined by declaring $E_+(S)=T_1$ and $E_-(S)=T_2$.
The Lee form $\theta$ decomposes accordingly as $\theta=\theta_++\theta_-$, where 
\begin{equation}\label{decompth}
\theta_+:=\frac{1}{2}(\theta+S\theta), \quad \theta_-:=\frac{1}{2}(\theta-S\theta).
\end{equation}
\end{remark}

For later use, we recall the action of the Riemannian curvature tensor on $S$ obtained in \cite[Equality (7)]{mp2025}, which follows in a straightforward way from \eqref{derivS} applied twice to vector fields $X,Y\in\Gamma(\T M)$:
\begin{equation}\label{RS}
	\begin{split}
		R_{X,Y} S	=&SY \odot  TX-SX \odot  TY+STY\odot X-STX\odot Y\\
		&+\<\theta, Y\>(SX\odot \theta-S\theta \odot X)+\<\theta, X\>(S\theta \odot Y-SY\odot \theta)\\
		& -\|\theta\|^2(SX\odot Y-SY\odot X).
	\end{split}	
\end{equation}	

Our framework throughout the paper is summarized in the following:

\begin{hypothesis}\label{hypo}
$(M,g)$ is an oriented, compact, connected $n$-dimensional Riemannian manifold, and $D$ is a Weyl connection different from the Levi-Civita connection $\nabla$ of $g$, such that:
\vspace{-0.2cm}
\begin{itemize}
\item  The Weyl connection $D$ has reducible holonomy, so it defines a conformal product structure, and $S$ denotes the corresponding $D$-parallel orthogonal involution given by Remark~\ref{remcps}.
\item The metric $g$ has reducible holonomy and $P$ denotes the corresponding $\nabla$-parallel orthogonal involution given by Remark~\ref{remgred}.
\end{itemize}
\end{hypothesis}

Let us denote by $\theta$ the Lee form of $D$ with respect to $g$ (which is not identically zero by assumption), and by \mbox{$T:=\nabla\theta$} its covariant derivative. We will identify $T$ with the corresponding endomorphism of $\T M$ using the metric $g$.

The exterior 2-form $\alpha\in\Omega^2(M)$, identified with the skew-symmetric endomorphism of  $\T M$ defined by
\begin{equation}\label{defalpha}
\alpha:=S\J-\J S
\end{equation}
will be of particular interest in the sequel. We derive some of its basic properties in the next result:

\begin{lemma}\label{lemmaalpha}
	The following identities hold:
	\begin{equation}\label{derivalpha}
		\nabla_X\alpha= \J S X\wedge\theta+\J \theta\wedge SX-\J X\wedge S\theta-\J S\theta\wedge X, \qquad \forall X\in \T M.
	\end{equation}	
	\begin{equation}\label{diffalpha}
		\di \alpha=-\alpha\wedge\theta,
	\end{equation}	
	\begin{equation}\label{codiffalpha}
		\delta\alpha=(1-n)\J S\theta-S\J \theta-\tr(\J S)\theta+\tr(\J)S\theta+\tr(S)\J\theta,
	\end{equation}	
	\begin{equation}\label{dtrace}
		\di(\tr(S\J))=2\alpha(\theta).
	\end{equation}	
\end{lemma}

\begin{proof}
	The identities follow by straightforward computation. For any vector fields $X,Y$, we obtain, using the fact that $\J$ is $\nabla$-parallel and the covariant derivative of $S$ is explicitly given by \eqref{derivS}:
	\begin{equation*}
		\begin{split}
			(\nabla_X\alpha)(Y)=& (\nabla_X(S\J-\J S))(Y)=(\nabla_X S)\J Y-\J(\nabla_X S)(Y)\\
			=&(SX\odot \theta - S\theta\odot X)\J Y-\J(SX\odot \theta - S\theta\odot X)(Y)\\
			=&\<SX, \J Y\>\theta +\theta(\J Y)SX - \J S\theta (Y)X-\<X, \J Y\>S\theta\\
			&-\<SX, Y\>\J \theta - \theta(Y)\J S X+S\theta(Y)\J X+\<X, Y\>\J S\theta\\
			=&(\J S X\wedge\theta+\J \theta\wedge SX-\J X\wedge S\theta-\J S\theta\wedge X)(Y),
		\end{split}
	\end{equation*}
	which proves \eqref{derivalpha}. If $\{e_i\}_{i=\overline{1,n}}$ is a local orthonormal basis of $\T M$, then we compute:
	\begin{equation*}
		\begin{split}
			\di \alpha&=\sum_{i=1}^n e_i\wedge \nabla_{e_i}\alpha\overset{\eqref{derivalpha}}{=}\sum_{i=1}^n e_i\wedge \left(\J S e_i\wedge\theta+\J \theta\wedge Se_i-\J e_i\wedge S\theta-\J S\theta\wedge e_i\right)\\
			&=\sum_{i=1}^n e_i\wedge \J S e_i\wedge\theta=-\alpha\wedge\theta,
		\end{split}
	\end{equation*}
	\begin{equation*}
		\begin{split}
			\delta\alpha&=-\sum_{i=1}^n e_i\lrcorner \nabla_{e_i}\alpha\overset{\eqref{derivalpha}}{=}-\sum_{i=1}^n e_i\lrcorner \left(\J S e_i\wedge\theta+\J \theta\wedge Se_i-\J e_i\wedge S\theta-\J S\theta\wedge e_i\right)\\
			&=-\tr(\J S)\theta+\J S\theta-S\J\theta+\tr(S)\J\theta+\tr(\J)S\theta-\J S\theta+(1-n)\J S\theta,
		\end{split}
	\end{equation*}
	yielding \eqref{diffalpha} and \eqref{codiffalpha}. In order to obtain \eqref{dtrace}, we compute for any vector field $X$ taking the orthonormal basis $\{e_i\}_{i=\overline{1,n}}$ to be parallel at the point where the computation is done:
	\begin{equation*}
		\begin{split}
			\di(\tr(S\J))(X)&=X(\tr(S\J))=\sum_{i=1}^n X(\<S\J e_i, e_i\>)=\sum_{i=1}^n X(\<\J S e_i, e_i\>)=\sum_{i=1}^n \<\J (\nabla_X S) e_i, e_i\>)\\
			&\overset{\eqref{derivS}}{=}\sum_{i=1}^n \<(SX\odot \theta - S\theta\odot X)(e_i) , \J e_i\>)\\
			&=\sum_{i=1}^n \left(\<SX,e_i\>\<\theta, \J e_i\>+\theta(e_i)\<SX, \J e_i\> - \<S\theta, e_i\>\<X, \J e_i\>-\<X, e_i\>\<S\theta, \J e_i\>\right)\\
			&=\<SX,\J\theta\>+\<\J SX,\theta\>-\<S\theta,\J X\>-\<X,\J S\theta\>=2\<\alpha(\theta), X\>.
		\end{split}
	\end{equation*}
\end{proof}

\section{Triple products}

We start with the following construction:
\begin{example}\label{etp}
	Let $(M_i, g_i)$, for $i\in\{1,2,3\}$, be three Riemannian manifolds and consider a function   $f\in\mathcal{C}^\infty(M_1\times M_2)$. On the product $M:=M_1\times M_2\times M_3$ we define the Riemannian metric $g:=e^{2f} g_1+g_2+g_3$. We remark that the manifold $M$ is endowed simultaneously with a Riemannian metric and with a Weyl connection, both having reducible holonomy, i.e. satisfying Hypothesis \ref{hypo}:
	\begin{enumerate}[label=(\roman*)]
		\item The Riemannian manifold $(M,g)$ is the Riemannian product  of $(M_1\times M_2,e^{2f} g_1+g_2)$ with $(M_3,g_3)$, so the holonomy of $g$ is reducible.
		\item The $1$-form $\theta:=-\di_2 f$ defines a Weyl connection $D$ on $(M,g)$, which together with the orthogonal $D$-parallel decomposition $\T M= \T M_1\oplus (\T M_2\oplus \T M_3)$ defines a conformal product structure on $(M,g)$, meaning that the holonomy of $D$ is reducible (cf. \cite{bm2011} or \cite[Prop. 3.2]{mp2025}). Here $\di_2$ denotes the differential of functions on $M$ along $M_2$. 
	\end{enumerate}	
\end{example}

The previous example motivates the following:

\begin{defn}\label{deftp}
	A Riemannian manifold $(M,g)$ is called a {\it triple product} if the tangent bundle of $M$ decomposes into an orthogonal direct sum of three integrable distributions \mbox{$\T M=T_1\oplus T_2\oplus T_3$}, such that every point of $M$ has a neighbourhood $U$ which can be written as $M_1\times M_2\times M_3$, with $T_i=\T M_i$ along $U$, and the metric $g$ restricted to $U$ takes the form $e^{2f} g_1+g_2+g_3$, where $g_i$ is a Riemannian metric on $M_i$, for $i\in\{1,2,3\}$, and $f\in\mathcal{C}^\infty(M_1\times M_2)$.
\end{defn}	

\begin{remark}
	This definition generalizes the notion of triple warped product introduced in~\cite{m2019}, where the manifold is considered to be a global product of three manifolds and, in the notation of  Definition~\ref{deftp}, the function $f$ is assumed to be a function on $M_2$ alone.
\end{remark}	

We now provide a tensorial characterization of triple products.
\begin{prop}\label{proptriple}
	Let $M$ be a manifold satisfying Hypothesis~\ref{hypo}. If the following two conditions are satisfied:
	\begin{enumerate}[label=(\roman*)]
		\item $S\J=\J S$,
		\item $E_{\varepsilon}(S)\subseteq E_{\varepsilon'}(\J)$, for some $\varepsilon, \varepsilon'\in\{-1,1\}$,
	\end{enumerate}	
	then $(M,g)$ is a triple product.
\end{prop}	

\begin{proof}
	Up to changing the signs of $S$ and $\J$ if necessary, we may assume without loss of generality that $\varepsilon=\varepsilon'=1$ in $(ii)$, i.e. that $E_{+}(S)\subseteq E_{+}(\J)$. Since $S$ and $\J$ are orthogonal involutions, the inclusion $E_{+}(S)\subseteq E_{+}(\J)$ implies that also the inclusion $E_{-}(\J)\subseteq E_{-}(S)$ holds. Moreover, since $S$ and $\J$ commute, $S$ preserves $E_{+}(\J)$, which thus decomposes as $E_{+}(\J)=E_{+}(S)\oplus (E_{+}(\J)\cap E_{-}(S)).$
	We thus have the following orthogonal direct sum decompositions of the tangent bundle:
	$$\T M=E_{+}(\J)\oplus E_{-}(\J) =E_{+}(S)\oplus (E_{+}(\J)\cap E_{-}(S))\oplus E_{-}(\J)=E_{+}(S)\oplus E_{-}(S).$$
Because both the Levi-Civita and the Weyl connection $D$ are torsion-free, the three distributions $E_{+}(S)$,  $E_{+}(\J)\cap E_{-}(S)$, and $E_{-}(\J)$ are integrable. Consequently, every point of $M$ has a neighbourhood $U$ which can be written as the product of three connected manifolds $M_1\times M_2\times M_3$, such that $\T M_1= E_{+}(S)$, $\T M_2=E_{+}(\J)\cap E_{-}(S)$ and $\T M_3= E_{-}(\J)$.

Now, because $\J$ is the orthogonal involution defined by the holonomy reduction of the metric $g$, the restriction of $g$ to this neighbourhood can be written as 
	\begin{equation}\label{gred}
		g=g_{12}+g_3,
	\end{equation}
	where $g_{12}$ is a metric on $M_1\times M_2$ and  $g_3$ is a metric on $M_3$.
	Furthermore, because $S$ is the orthogonal involution corresponding to the conformal product structure, after shrinking $M_1$, $M_2$ and $M_3$ if necessary, the restriction of $g$ to the neighbourhood $U$ can also be expressed as
	\begin{equation}\label{gcps}
		g=e^{f_1}h_{1}+e^{f_{23}}h_{23},
	\end{equation}
	where $f_1, f_{23}\in\mathcal{C}^\infty(M_1\times M_2\times M_3)$, $h_1$ is a metric on $M_1$ and $h_{23}$ is a metric on $M_2\times M_3$ (cf. \cite[Prop. 3.2]{mp2025}).  From the expressions \eqref{gred} and \eqref{gcps} of the metric $g$, we obtain that the metric $h_{23}$ can be decomposed as $h_{23}(y,z)=h_2(y,z)+h_3(y,z)$, where $h_2(\cdot ,z)$ is a metric on $M_2\times\{z\}$  for every $z\in M_3$ and $h_3(y,\cdot)$ is a metric on $M_3$  for every $y\in M_2$. It follows in particular that $g_3(z)=e^{f_{23}(x,y,z)}h_3(y,z)$ for every $(x,y,z)\in M_1\times M_2\times M_3$, hence $f_{23}\in\mathcal{C}^\infty(M_2\times M_3)$. If we further compare the restrictions of \eqref{gred} and \eqref{gcps}  to $M_1\times M_2$ we obtain that
	\begin{equation}\label{g12}
		g_{12}(x,y)=e^{f_1(x,y,z)}h_1(x)+e^{f_{23}(y,z)}h_{2}(y,z),\qquad\forall (x,y,z)\in M_1\times M_2\times M_3.
	\end{equation}
	Differentiating the Equality \eqref{g12} in the direction of a vector field $Z\in\Gamma (\T M_3)$, we obtain:
	\begin{equation}\label{g12diff}
		0=Z(f_1)e^{f_1}h_1+\mathcal{L}_Z(e^{f_{23}}h_{2}),
	\end{equation}
	where  $Z(f_1)e^{f_1}h_1$ belongs to $\Sym^2(M_1)$ and $\mathcal{L}_Z(e^{f_{23}}h_{2})$ belongs to $\Sym^2(M_2\times M_3)$, implying thus that each of the terms  in the right-hand side of \eqref{g12diff} has to vanish. Thus $Z(f_1)=0$, for all vector fields $Z$ tangent to $M_3$, implying that $f_1\in\mathcal{C}^\infty(M_1\times M_2)$, and also $\mathcal{L}_Z(e^{f_{23}}h_{2})=0$ for all $Z\in\Gamma (\T M_3)$, which implies that $e^{f_{23}}h_{2}=:g_2$ is a metric on $M_2$ which does not depend on $M_3$. Altogether, we have showed that the metric $g$ can be written as follows:
	$$g=e^{f_1}g_1+g_2+g_3,$$
	with $g_1:=h_1$. By the above considerations, $g_i$ is a metric on $M_i$, for $i\in\{1,2,3\}$, and $f_1\in\mathcal{C}^\infty(M_1\times M_2)$, so $(M, g)$ is a triple product.
\end{proof}	

\section{A scalar identity for conformal product structures on reducible manifolds}

The purpose of this section is to prove the following result:
\begin{prop}\label{propscalarformula}
	Let $M$ be a manifold satisfying Hypothesis~\ref{hypo}. Then the following identity holds:
	\begin{equation}\label{traceidentity2}
		\begin{split}
			0=&\frac12\|\theta\|^2(n^2+(\tr(S))^2-(\tr(\J))^2-(\tr(S\J))^2+\tr(S\J S\J)-n)\\
			&+\<\theta, \J\theta-S\J S\theta\>\tr(\J)-\<\theta, S\theta\>(n\tr(S)-\tr(\J)\tr(S\J))	+\delta(\beta),
		\end{split}	
	\end{equation}
	where $\beta:=\tr(S) S\theta-(n+1)\theta-\tr(S\J)\J S\theta+\J S\J S\theta+\tr(U)\J \theta\in \Omega^1(M)$.
\end{prop}	

This formula will be obtained roughly speaking  by taking traces in an identity obtained by applying the Riemannian curvature tensor of $g$ to $S$ and using its commutation with $\J$. In order to obtain this identity we start in a more general setting with the following definition:

\begin{definition}\label{defcdot}
	For any endomorphisms $F, G\in\mathrm{End}(\T M)$, we define $F\cdot G\in \mathrm{End}(\Lambda^2 (\T M))$ by the following formula:
$$(F\cdot G) (X\wedge Y):=\frac{1}{2}(F(X)\wedge G(Y)+G(X)\wedge F(Y)), \text { for all tangent vectors } X,Y.$$
\end{definition}

Note that if $F,G$ are symmetric endomorphisms of $\T M,g)$, then $F\cdot G$ is up to a constant factor exactly the Kulkarni-Nomizu product of $F$ and $G$. 
It follows directly from the above definition that $(F\cdot F)(X\wedge Y)=F(X)\wedge F(Y)$ and $F\cdot G=G\cdot F$, for all $F, G\in \mathrm{End}(\T M)$. We will need the following general identities:
\begin{lemma}\label{genlem}
	For any $F, G, F',G'\in\mathrm{End}(\T M)$ the following relations hold:
	\begin{enumerate}[label=(\roman*)]
	  \item $\tr(F\cdot G)=\displaystyle\frac{1}{2}\left(\tr(F)\tr(G) -\tr(F\circ G)\right).$
		\item $(F\cdot G)\circ (F'\cdot G')=\displaystyle\frac{1}{2} ((F\circ F')\cdot (G\circ G')+(G\circ F')\cdot(F\circ G')).$
		In particular, for $F'=G'$, we obtain:
 \begin{equation}\label{prod}
 (F\cdot G)\circ (F'\cdot F')=(F\circ F')\cdot (G\circ F').
 \end{equation}
 	\end{enumerate}	
\end{lemma}	
\begin{proof}
	$(i)$ If $\{e_i\}_{1\leq i\leq n}$ is a local orthonormal basis of $\T M$, then $\{e_i\wedge e_j\}_{1\leq i<j\leq n}$ is a local orthonormal basis of $\Lambda^2 (\T M)$, and the trace of the endomorphism $F\cdot G\in \mathrm{End}(\Lambda^2 (\T M))$ can be computed as follows:
	\begin{equation*}
		\begin{split}
			\tr(F\cdot G)&=\frac{1}{2}\sum_{i,j=1}^n \<(F\cdot G)(e_i\wedge e_j),e_i\wedge e_j\>=\frac{1}{4}\sum_{i,j=1}^n \<F(e_i)\wedge G(e_j)+G(e_i)\wedge F(e_j),e_i\wedge e_j\>\\
			&=\frac{1}{2}\sum_{i,j=1}^n \left(\<F(e_i), e_i\>\<G(e_j),e_j\>-\<F(e_i), e_j\>\<G(e_j),e_i\>\right)\\
			&=\frac{1}{2}\left(\tr(F)\tr(G)-\sum_{j=1}^n \<F(G(e_j)), e_j\>\right)=\frac{1}{2}\left(\tr(F)\tr(G) -\tr(F\circ G)\right).
		\end{split}	
	\end{equation*}
$(ii)$ We compute by definition, for any tangent vectors $X$ and $Y$:
\begin{equation*}
	\begin{split}
		4((F\cdot G)\circ (F'\cdot G'))(X\wedge Y)=&2(F\cdot G)(F'(X)\wedge G'(Y)+G'(X)\wedge F'(Y))\\
		=&F(F'(X))\wedge G(G'(Y))+G(F'(X))\wedge F(G'(Y))\\
		&+F(G'(X))\wedge G(F'(Y))+G(G'(X))\wedge F(F'(Y))\\
		=&2((F\circ F')\cdot (G\circ G')+(G\circ F')\cdot(F\circ G'))(X\wedge Y).
	\end{split}	
\end{equation*}
\end{proof}	

 We assume now that $M$ satisfies Hypothesis \ref{hypo} and determine the compositions of $\J\cdot \J$ and $S\cdot S$ with the Riemannian curvature tensor.

\begin{lemma}\label{lemmaRSJ}
	If $R\colon\Lambda^2 (\T M)\to \Lambda^2 (\T M)$ denotes the Riemannian curvature tensor of the metric $g$, seen as a symmetric endomorphism of  $ \Lambda^2 (\T M)$, then the following identities hold:
	\begin{equation}\label{commJR}
(\J\cdot \J)\circ R-R=0=R\circ (\J\cdot \J)-R,
	\end{equation}
	\begin{equation}\label{commSR}
	(S\cdot S)\circ R-R=-2ST\cdot S+2T\cdot I+2S\cdot(\theta\otimes S\theta)-2I\cdot(\theta\otimes \theta)+\|\theta\|^2(I\cdot I-S\cdot S),
\end{equation}
where $T:=\nabla\theta$ and $I:=\Id_{\T M}$ denotes the identity endomorphism of $\T M$.
\end{lemma}	

\begin{proof}
	Since $\J$ is a $\nabla$-parallel symmetric involution of $\T M$ and $\J\cdot\J$ is symmetric with respect to the induced metric on $\Lambda^2 (\T M)$, we obtain for all tangent vector fields $X,Y,Z,W$:
	\begin{equation*}
		\begin{split}
			\<((\J\cdot \J)\circ R)(X\wedge Y), Z\wedge W\>&=\<R(X\wedge Y), (\J\cdot \J)(Z\wedge W)\>=\<R(X\wedge Y), \J Z\wedge \J W\>\\
			&=\<R_{X,Y}\J Z, \J W\>\overset{\nabla\J=0}{=}\<\J R_{X,Y} Z, \J W\>\\
			&=\<R_{X,Y}Z, W\>=\<R(X\wedge Y), Z\wedge W\>,
		\end{split}	
	\end{equation*}
yielding the first equality in \eqref{commJR}. The second equality  in \eqref{commJR} then follows by taking the transpose of the first equality, since both $R$ and $\J\cdot \J$ are symmetric endomorphisms.

In order to prove \eqref{commSR}, we start similarly by computing for all tangent vector fields $X,Y,Z,W$:
\begin{equation}\label{ssr}
\begin{split}
		\<((S\cdot S)\circ R)(X\wedge Y), Z\wedge W\>&=\<R(X\wedge Y), (S\cdot S)(Z\wedge W)\>=\<R_{X,Y}S Z, S W\>\\
		&=\<(R_{X,Y}S) Z, S W\>+\<S R_{X,Y}Z, S W\>\\
		&=\<(S\circ R_{X,Y}S )Z, W\>+\<R(X\wedge Y), Z\wedge W\>.
\end{split}		
\end{equation}
Using \eqref{RS} composed with $S$ on the left and regrouping the terms in order to use the operation introduced in Definition~\ref{defcdot}, we obtain:
\begin{equation*}
	\begin{split}
		S\circ R_{X,Y} S	=&SY \otimes  STX+TX\otimes Y -SX \otimes STY-TY\otimes X+  STY\otimes SX+X\otimes TY\\
		&-STX\otimes SY-Y\otimes TX+\<\theta, Y\>(SX\otimes S\theta+\theta\otimes X-S\theta \otimes SX + X\otimes \theta)\\
		&+\<\theta, X\>(S\theta \otimes SY+Y\otimes \theta-SY\otimes S\theta+\theta\otimes Y)\\
		& -\|\theta\|^2(SX\otimes SY+Y\otimes X-SY\otimes SX- X\otimes Y)\\
		=&-STX\wedge SY- SX\wedge STY+TX\wedge Y+X\wedge TY+\<\theta, Y\>(SX\wedge S\theta+\theta\wedge X)\\
		&+\<\theta, X\>(S\theta \wedge SY+Y\wedge \theta) -\|\theta\|^2(SX\wedge SY-X\wedge Y)\\
		=&- 2 (ST\cdot S)(X\wedge Y)+2 (T\cdot I) (X\wedge Y)+SX\wedge (\theta\otimes S\theta)(Y)-X\wedge (\theta\otimes\theta)(Y)\\
		&+(\theta\otimes S\theta)(X)\wedge SY- (\theta\otimes \theta)(X)\wedge Y+\|\theta\|^2(I\cdot I-S\cdot S)(X\wedge Y)\\
		=&( -2 ST\cdot S+2 T\cdot I +2S\cdot (\theta\otimes S\theta)-2I\cdot (\theta\otimes\theta)+\|\theta\|^2(I\cdot I-S\cdot S))(X\wedge Y),
	\end{split}	
\end{equation*}	
which together with \eqref{ssr} yields \eqref{commSR}.
\end{proof}	

With the same notation as in Lemma \ref{lemmaRSJ} we may now formulate the following direct consequence:

\begin{corollary}\label{corol}
The endomorphism $(S\cdot S)\circ R-R$ can be expressed as follows:
\begin{equation}\label{eqSSR}
	\begin{split}
		(S\cdot S)\circ R-R\overset{\eqref{commJR}}{=}&((S\cdot S)\circ R-R) \circ (\J\cdot \J)\\
		\overset{\eqref{commSR}, \eqref{prod}}{=}&-2ST\J\cdot S\J+2T\J \cdot \J+2S\J\cdot(\J\theta\otimes S\theta)\\
		&-2\J\cdot(\J\theta\otimes \theta)+\|\theta\|^2(\J\cdot \J-S\J\cdot S\J).
	\end{split}	
\end{equation}
\end{corollary}
From the above formulas we obtain:

\noindent{\bf Proof of Proposition~\ref{propscalarformula}.}
	Taking the trace in both expressions obtained for the endomorphism $(S\cdot S)\circ R-R$, namely in \eqref{commSR} and \eqref{eqSSR}, and applying Lemma~\ref{genlem} (i), yields after some simplification:
	\begin{equation}\label{traceidentity0}
		\begin{split}
			&-\tr(ST)\tr(S) +n\tr(T)+\tr(S)\<\theta, S\theta\>+\frac12\|\theta\|^2(n^2-2n-(\tr(S))^2)=\\
			&-\tr(\J ST) \tr(S\J)+\tr(\J S\J S T)+\tr(\J T)\tr(\J)-\tr(T)+\tr(S\J)\<\J\theta, S\theta\>\\
			&-\<\J\theta, S\J S\theta\>
			-\tr(\J)\<\theta, \J \theta\>+\frac12\|\theta\|^2(2+(\tr(\J))^2-n-(\tr(S\J))^2+\tr(S\J S\J)).
		\end{split}	
	\end{equation}
	We first express separately the term $\tr(\J ST)\tr(S\J)$ occuring in this equality, because the trace of the endomorphism $S\J$ is  not necessarily constant. For this, we start by computing the following codifferential:
	\begin{equation}\label{noncttr}
		\begin{split}
			-\delta(\tr(S\J)\J S\theta)&=\ -\tr(S\J)\delta(\J S\theta)+\<\J S\theta, \di(\tr(S\J))\>\\
			&\!\!\overset{\eqref{dtrace}}{=}-\tr(S\J)\delta(\J S\theta)+2\<\J S\theta, (S\J-\J S)\theta\>\\
			&=\ -\tr(S\J)\delta(\J S\theta)+2\<S\J S\J \theta, \theta\>+2\|\theta\|^2,
		\end{split}
	\end{equation}
	which together with the computations in the Appendix yields:
	\begin{equation}\label{noncttr1}
		\begin{split}
			&\tr(\J ST)\tr(S\J)\overset{\eqref{trJST}}{=}(-\delta(\J S\theta)-\|\theta\|^2\tr(S\J)+\<\theta, S\theta\>\tr(\J))\tr(S\J)\\
			&\overset{\eqref{noncttr}}{=} -\delta(\tr(S\J)\J S\theta)-2\<S\J S\J \theta, \theta\>+2\|\theta\|^2-\|\theta\|^2(\tr(S\J))^2+\<\theta, S\theta\>\tr(\J)\tr(S\J).
		\end{split}
	\end{equation}
If we replace the traces computed in the Appendix, Lemma \ref{traces}, as well as the identity \eqref{noncttr1}, in Equality~\eqref{traceidentity0}, we then obtain:
	\begin{equation}\label{traceidentity1}
		\begin{split}
			&\|\theta\|^2(\tr(S))^2-n\<\theta, S\theta\>\tr(S)+\<\theta, S\theta\>\tr(S)+\frac12\|\theta\|^2(n^2-2n-(\tr(S))^2)\\
			&-2\<S\J S\J\theta, \theta\>+2\|\theta\|^2-\|\theta\|^2(\tr(S\J))^2+\<\theta, S\theta\>\tr(\J)\tr(S\J) \\
			& -\|\theta\|^2(1-\tr(\J S\J S))+\<\theta, \J S\J S\theta\>\\
			&+\<\J \theta, S\theta\>\tr(\J S)-\<\theta, S\J S\theta\>\tr (\J)-\<\theta, S\theta\>\tr(S)\\
			&-\tr(S\J)\<\J\theta, S\theta\>+\<\theta, S\J S\J\theta\>+\tr(\J)\<\theta, \J\theta\>\\
			&-\frac12\|\theta\|^2(2+(\tr(\J))^2-n-(\tr(S\J))^2+\tr(S\J S\J))+\delta(\beta)=0,
		\end{split}	
	\end{equation}
	where $\beta:=\tr(S) S\theta-(n+1)\theta-\tr(S\J)\J S\theta+\J S\J S\theta+\tr(\J)\J \theta$. Regrouping the terms of the same kind, Equation~\eqref{traceidentity1} yields \eqref{traceidentity2} and finishes the proof of Proposition~\ref{propscalarformula}.\qed

\section{An integral formula}

We assume throughout this section that $M$ satisfies Hypothesis \ref{hypo}. In order to state the next results, we introduce the following  two functions on $M$:
\begin{equation}\label{defa+}
	A_+:=n^2+(\tr(S))^2-(\tr(\J))^2-(\tr(S\J))^2+\tr(S\J S\J)-n-2n\tr(S)+2\tr(\J)\tr(S\J),
\end{equation}
\begin{equation}\label{defa-}
	A_-:=n^2+(\tr(S))^2-(\tr(\J))^2-(\tr(S\J))^2+\tr(S\J S\J)-n+2n\tr(S)-2\tr(\J)\tr(S\J).
\end{equation}

\begin{remark}\label{a+a-const}
Let us notice that in general the functions $A_+$ and $A_-$ are not constant on $M$, because although involutions of $\T M$ have constant trace on $M$, the traces of $S\J$  and $S\J S\J$ are not necessarily constant. However, if  $S$ and $\J$ commute, then $A_+$ and $A_-$ are constant, because in this case $S\J$ is an involution itself, so $\tr(S\J)$ is constant, and $S\J S\J=I$, so $\tr(S\J S\J)=n$.
\end{remark}	

We denote by $\vol^g$ the volume form of $(M,g)$ and consider the following condition, which will be referred to in the sequel as the condition~$(*)$:
\begin{equation}
	\displaystyle\tr(\J) \cdot \int_M\<\theta, \J\theta-S\J S\theta\>\vol^g=0\tag{$\mathcal{*}$}.
\end{equation}
Let us first show that this condition holds under various assumptions:
\begin{lemma} \label{lemma*}
	Each of the following assumptions:
	\begin{enumerate}[label=(\roman*)]
		\item $S\theta=\varepsilon\theta$, for some $\varepsilon\in\{-1,1\}$,
		\item $\delta S\theta=\varepsilon\delta\theta$,  for some $\varepsilon\in\{-1,1\}$,
		\item $S\J=\J S$,
		\item $\tr (S\J)$ is constant,
		\item $\delta(S\theta)=0$,
		\item $\tr(\J)=0$,
	\end{enumerate}	
	implies that condition $(*)$ is satisfied.
\end{lemma}
\begin{proof}
	It is clear that $(i)$ implies $(ii)$ and that $(iii)$ implies $(iv)$. It is also obvious that $(vi)$ implies condition~$(*)$.\\
	If $(iv)$ holds, then the endomorphism $\alpha$ defined by~\eqref{defalpha}~satisfies $\displaystyle\alpha(\theta)\overset{\eqref{dtrace}}{=}\frac{1}{2}\di(\tr(S\J))=0$, so the integrand in condition $(*)$ vanishes:  $\<\theta, \J\theta-S\J S\theta\>=\<S\theta, \alpha(\theta)\>=0$.\\
 Using Stokes' Theorem we compute:
		\begin{equation}\label{eqstokes}
		\begin{split}
			\int_M\<\theta, \J\theta-S\J S\theta\>\vol^g&=\int_M\<S\theta, \alpha(\theta)\>\vol^g\overset{\eqref{dtrace}}{=}\frac{1}{2}\int_M\<S\theta, \di(\tr(S\J))\>\vol^g\\
			&=\frac{1}{2}\int_M\delta(S\theta)\cdot \tr(S\J)\vol^g.
		\end{split}	
	\end{equation}
	If $(v)$ holds, then \eqref{eqstokes} yields condition $(*)$.
	Finally, if $(ii)$ holds, then we compute further in \eqref{eqstokes} to obtain:
	\begin{equation*}
		\begin{split}
		\int_M\<\theta, \J\theta-S\J S\theta\>\vol^g&=\frac{1}{2}\int_M\delta(S\theta)\cdot \tr(S\J)\vol^g\overset{(ii)}{=}\frac{1}{2}\int_M\delta(\theta)\cdot \tr(S\J)\vol^g\\
		&=\frac{1}{2}\int_M\<\theta, \di(\tr(S\J))\>\vol^g\overset{\eqref{dtrace}}{=}\int_M\<\theta, \alpha(\theta))\vol^g=0,
	\end{split}	
	\end{equation*}
where the last equality follows from the fact that $\alpha$ is skew-symmetric.
\end{proof}	

\begin{lemma}\label{lemmaintegral}
If the condition $(*)$ is satisfied and  $\theta=\theta_++\theta_-$ is the decomposition given by~\eqref{decompth}, then the following integral vanishes: 
	\begin{equation}\label{integral}
		\int_M (A_+\|\theta_+\|^2+A_-\|\theta_-\|^2)\vol^g=0,
	\end{equation}
	where the functions $A_+$ and $A_-$ are defined by \eqref{defa+} and \eqref{defa-}.
\end{lemma}
\begin{proof}
	Under the assumption $(*)$, integrating Equality~\eqref{traceidentity2} over the compact manifold $M$ yields
	\begin{equation}\label{traceidentity3}
		\begin{split}
			0=&\displaystyle\int_M\|\theta\|^2(n^2+(\tr(S))^2-(\tr(\J))^2-(\tr(S\J))^2+\tr(S\J S\J)-n)\vol^g\\
			&-2\displaystyle\int_M\<\theta, S\theta\>(n\tr(S)-\tr(\J)\tr(S\J))\vol^g.
		\end{split}	
	\end{equation}
	Replacing  $\|\theta\|^2=\|\theta_+\|^2+\|\theta_-\|^2$ and $\<\theta, S\theta\>=\|\theta_+\|^2-\|\theta_-\|^2$ into Equality~\eqref{traceidentity3} yields \eqref{integral}, where the coefficients $A_+$ and $A_-$ are given by  \eqref{defa+} and \eqref{defa-}. 
\end{proof}	

\begin{theorem}\label{thma+a-}
  The functions $A_+$ and $A_-$ defined by \eqref{defa+} and  \eqref{defa-} have the following properties:
	\vspace{-0.3cm}
	\begin{enumerate}[label=(\roman*)]
\item $A_+$ and $A_-$ are non-negative.
\item If  $\rk(E_+(S))=1$, then $A_-\equiv0$. If  $\rk(E_+(S))\geq 2$, then the following equivalence holds at each point $x\in M$: 
\begin{equation}\label{equiva-}
	A_-(x)=0 \Longleftrightarrow \exists\, \varepsilon\in\{-1,1\} \text{ such that } E_{+}(S)\subseteq E_{\varepsilon}(\J) \text{ at } x.
\end{equation}
\item If  $\rk(E_-(S))=1$, then $A_+\equiv0$.
If $\rk(E_-(S))\geq 2$, then the following equivalence holds at each point $x\in M$: 
\begin{equation}\label{equiva+}
	A_+(x)=0 \Longleftrightarrow \exists\, \varepsilon\in\{-1,1\} \text{ such that } E_{-}(S)\subseteq E_{\varepsilon}(\J) \text{ at } x.
\end{equation}
\end{enumerate}
\end{theorem}
\begin{proof} Let $r$ denote the rank of $E_+(S)$.\\
	$(i)$ 	If $\{\xi_i\}_{i=\overline{1,r}}$ is a local orthonormal basis of $E_+(S)$, then the endomorphisms occurring in \eqref{defa+} and \eqref{defa-} and their traces can be locally  expressed as follows:
	\begin{equation}\label{formulaS}
		S=-I+2\sum_{i=1}^r \xi_i\otimes \xi_i, \quad \tr(S)=2r-n,
	\end{equation}	
	\begin{equation}\label{formulaSJ}
		S\J=-\J+2\sum_{i=1}^r \J \xi_i\otimes \xi_i, \quad \tr(S\J)=2\sum_{i=1}^r\<\J\xi_i, \xi_i\>-\tr(\J),
	\end{equation}	
	\begin{equation}\label{formulaJSJ}
		\J S\J=-I+2\sum_{i=1}^r \J \xi_i\otimes \J \xi_i, \quad \tr(\J S\J)=\tr(S),
	\end{equation}	
	\begin{equation}\label{formulaSJSJ}
		\begin{split}	
			S\J S\J&=-S+2\sum_{i=1}^r \J \xi_i\otimes S \J \xi_i\\
			&\!\!\overset{\eqref{formulaSJ}}{=}-S+2\sum_{i=1}^r \J \xi_i\otimes \left(-\J \xi_i+2\sum_{j=1}^r\<\J \xi_i, \xi_j\>\xi_j\right)\\
			&=-S-2\sum_{i=1}^r \J \xi_i\otimes \J \xi_i+4\sum_{i,j=1}^r\<\J \xi_i, \xi_j\>\J\xi_i\otimes \xi_j,
		\end{split}
	\end{equation}	
	\begin{equation}\label{formulatrSJSJ}
		\begin{split}	
			\tr(S\J S\J)&\overset{\eqref{formulaS}}{=}n-4r+4\sum_{i,j=1}^r\<\J\xi_i,  \xi_j\>^2.
		\end{split}
	\end{equation}	
	Altogether, replacing these traces in  \eqref{defa-}, we obtain:
	\begin{equation*}
		\begin{split}
			A_-=&\ n^2+(n-2r)^2-(\tr(\J))^2-\left(2\sum_{i=1}^r\<\J\xi_i, \xi_i\>-\tr(\J)\right)^2+n-4r+4\sum_{i,j=1}^r\<\J\xi_i,  \xi_j\>^2\\
			&-n+2n(2r-n)-2\tr(\J)\left(2\sum_{i=1}^r\<\J\xi_i,  \xi_i\>-\tr(\J)\right)\\
			=&\ 4r^2-4r-4\left(\sum_{i=1}^r\<\J\xi_i, \xi_i\>\right)^2+4\sum_{i,j=1}^r\<\J\xi_i, \xi_j\>^2.
		\end{split}
	\end{equation*}		
	Thus, the coefficient $A_-$ is given by the following formula:
	\begin{equation}\label{a-}
		A_-=4\left(r^2-r+\sum_{\underset{i\neq j}{i,j=1}}^r\left(\<\J\xi_i, \xi_j\>^2-\<\J\xi_i, \xi_i\>\<\J\xi_j, \xi_j\>\right)\right).
	\end{equation}	
	For the coefficient $A_+$ there is similar formula, which can actually be deduced from \eqref{a-}. Namely, if we replace $S$ by $-S$, hence $r$ by $n-r$, then $A_+$ for $S$ becomes $A_-$ for $-S$ and thus we obtain:
\begin{equation}\label{a+}
	A_+=4\left((n-r)^2-(n-r)+\sum_{\underset{i\neq j}{i,j=1}}^{n-r}\left(\< \J\eta_i, \eta_j\>^2-\< \J\eta_i, \eta_i\>\< \J\eta_j, \eta_j\>\right)\right),
\end{equation}	
where $\{\eta_i\}_{i=\overline{1,n-r}}$ is a local orthonormal basis of $E_-(S)$.

	We notice that in \eqref{a-} the sum $\displaystyle\sum_{\underset{i\neq j}{i,j=1}}^r\left(\< \J\xi_i, \xi_j\>^2-\< \J\xi_i, \xi_i\>\< \J\xi_j,\xi_j\>\right)$ has exactly $r^2-r$ terms and each of them is greater or equal to $-1$ by the Cauchy-Schwarz inequality: $|\< \J \xi_i, \xi_i\>|\leq 1$, so $\< \J\xi_i, \xi_j\>^2-\< \J\xi_i, \xi_i\>\< \J\xi_j,\xi_j\>\geq -1$, for all distinct $i,j\in\{1,\dots, r\}$. Hence, $A_-\geq0$. The same argument applied to $-S$ instead of $S$ shows that $A_+\geq 0$.

$(ii)$ If $r=1$, then it follows directly from \eqref{a-} that the function $A_-$ vanishes identically.
	
Let $x\in M$. When $r\geq 2$, we know from the proof of $(i)$ that $A_-(x)=0$ if and only if in the point $x$ equality holds in all the above Cauchy-Schwarz inequalities, \emph{i.e.} $|\<\xi_i, \J \xi_i\>|=1$, and $\<\xi_i, \J \xi_i\>\<\xi_j, \J \xi_j\>=1$ for all $i,j\in\{1, \dots, r\}$, meaning that there exists $\varepsilon\in\{-1,1\}$ such that for all $i\in\{1,\dots, r\}$ we have $\J\xi_i=\varepsilon\xi_i$, so $E_{+}(S)\subseteq E_{\varepsilon}(\J)$ at $x$. This proves the equivalence \eqref{equiva-}.

$(iii)$	Follows from $(ii)$ by duality, if we consider $-S$ instead of $S$.
\end{proof}

\section{Conformal product structures on reducible compact Riemannian manifolds}

 In this  section we prove the following result, which together with Proposition~\ref{proptriple} and Lemma~\ref{lemma*} will imply Theorem~\ref{mainthm}:

\begin{theorem}\label{thmcond*}
	If $M$ is a compact manifold satisfying Hypothesis~\ref{hypo}, then the following assertions hold:
	\begin{enumerate}[label=(\roman*)]
		\item 	$S\J=\J S$.
		\item $E_{\varepsilon}(S)\subseteq E_{\varepsilon'}(\J)$, for some $\varepsilon, \varepsilon'\in\{-1,1\}$.
	\end{enumerate}	
\end{theorem}

In order to prove Theorem~\ref{thmcond*} we distinguish two cases, depending on the rank of the 
conformal product structure, namely if this rank is greater or equal to $2$, or if it is $1$.  This distinction is imposed by the equivalences obtained in Theorem~\ref{thma+a-}, $(ii)$ and $(iii)$.

{\bf Case 1. } We assume that the rank of the conformal product structure is greater or equal to $2$. 

$(i)$ We consider  the closed subset $C$ of $M$ where $\theta$ vanishes, $C:=\{x\in M\, |\, \theta(x)=0\}$.\\
We first show that $S$ and $\J$ commute on the complement $ M\setminus C$. For this let  $x\in M\setminus C$, \emph{i.e.} $\theta(x)\neq 0$. Hence, $\theta_+(x)\neq 0$ or $\theta_-(x)\neq 0$. By replacing $S$ with $-S$ if necessary, we may assume without loss of generality that the latter holds. Then \eqref{integral} yields that $A_-(x)=0$. Since the rank of the conformal product structure is assumed to be greater or equal to $2$, the equivalence \eqref{equiva+} implies that $E_{+}(S)\subseteq E_{\varepsilon}(\J)$ at the point $x$, for some $\varepsilon\in\{-1,1\}$. By replacing this time $\J$ with $-\J$ if necessary, we may assume without loss of generality that $\varepsilon=1$, so $E_{+}(S)\subseteq E_{+}(\J)$ at $x$. Looking at the orthogonal complements of these subspaces yields the inclusion $E_{-}(\J)\subseteq E_{-}(S)$ at $x$, hence
the tangent space at $x$ splits as an orthogonal direct sum $\T_x M=E_{+}(S)_x\oplus E_{-}(\J)_x\oplus K$, where $K$ is the orthogonal complement of $E_{-}(\J)_x$ in $E_{-}(S)_x$. On each of these three subspaces the endomorphisms $S$ and $\J$ act either as identity or minus identity, thus $S$ and $\J$ commute at the point $x$.\\
If the interior of $C$ is empty, then, since $S$ and $\J$ commute on $M\setminus C$, they commute everywhere on $M$.\\
If $\overset{\circ}{C}\neq \emptyset$, then by Lemma~\ref{lemmaalpha}, the function $\|S\J-\J S\|^2$ is constant on any connected component $V$ of $\overset{\circ}{C}$. Let $\overline{V}$ denote the closure of $V$ and $\partial V=\overline{V}\setminus V$ its boundary. If  $\partial V=\emptyset$, then $V$ is both open and closed, so the connectedness of $M$ yields $V=M$, whence $C=M$, which is impossible since $\theta$ is assumed not to vanish identically. Thus $\partial V\neq \emptyset$. Since $\partial V \subseteq \partial C =\partial(M\setminus C)$, for any point $x\in\partial V$ we can find two sequences $(x_n)_n$ with $x_n\in V$ and $(y_n)_n$ with $y_n\in M\setminus C$ and $\displaystyle \lim_{n\to+\infty}x_n= \displaystyle \lim_{n\to+\infty}y_n=x$. Hence, because  the function $\|S\J-\J S\|^2$ is continuous, it  vanishes on $M\setminus C$ and is constant on $V$, it follows that it also vanishes at $x$ and its constant value on $V$ must be zero. This proves that $S$ and $\J$ commute everywhere on $M$ in this case as well.

$(ii)$ Since $S$ and $\J$ commute, we know by Remark~ \ref{a+a-const} that the coefficients $A_+$ and $A_-$ are constant on $M$. Then Equality \eqref{integral} implies that at least one of $A_+$ or $A_-$ must vanish, which according to the equivalences \eqref{equiva-} and \eqref{equiva+} yields $(ii)$.

{\bf Case 2. } We assume that the conformal product structure has rank $1$.

In order to fix the notation, we assume that $\rk(E_-(S))=1$ (otherwise we replace $S$ by $-S$). Up to taking a finite cover, we can assume that the conformal product structure is orientable, so both $D$-parallel distributions $E_+(S)$ and $E_-(S)$ are orientable and  we choose a unit length vector field $\xi$ spanning the $1$-dimensional distribution $E_-(S)$.

We start by a few general results about conformal product structures of rank $1$, that will be needed in our proof. Recall first the following result, which was proven in \cite[ Lemma 2.7]{mp2025a}:

\begin{lemma}\label{rank1xi}
	The unit length vector field $\xi$ spanning the $D$-parallel $1$-dimensional distribution satisfies:
	\begin{equation}\label{koxi}
		\nabla_X\xi=-\theta(\xi)X+\langle X, \xi\rangle\theta,\qquad\forall X\in \Gamma(\T M),
	\end{equation}		
	and in particular:
	\begin{equation}\label{codiffxi}
		\nabla_\theta\xi=0,\qquad	\delta\xi=(n-1)\theta(\xi), \qquad \di\xi=\xi\wedge \theta.
	\end{equation}		
\end{lemma}

A direct consequence of  Lemma~\ref{rank1xi} is the following computation for any vector field $X$:
\begin{equation*}
	\begin{split}
		X(\<\xi, \J\xi\>)&=\<\nabla_X\xi, \J\xi\>+\<\xi, \J\nabla_X\xi\>\\
		&=-\theta(\xi)\<X, \J \xi\>+\<X, \xi\>\theta(\J \xi)-\theta(\xi)\<X, \J \xi\>+\<X, \xi\>\theta(\J\xi)\\
		&=2\theta(\J\xi)\<\xi, X\>-2\theta(\xi)\<\J\xi, X\>,
	\end{split}
\end{equation*}	
which yields the equality:
\begin{equation}\label{eqxiJxi}
	\di(\<\xi, \J\xi\>)=2\left(\theta(\J\xi)\xi-\theta(\xi)\J\xi\right).
\end{equation}
\begin{lemma}\label{lemmaric}
 If $\Ric$ denotes the Ricci curvature of the metric $g$, seen as a symmetric endomorphism of $\T M$, then the following equalities hold:
\begin{equation}\label{ricxi}
	\Ric(\xi)=(n-2)\di(\theta(\xi))-(\delta\theta)\xi,
\end{equation}		
\begin{equation}\label{ricJxi}
	\Ric(\J\xi)=\theta(\J \xi)\theta-\theta(\xi)\J\theta-(\delta\J\theta)\xi-\di(\theta(\J\xi))-\J\di(\theta(\xi))+\tr(\J)\di(\theta(\xi)).
\end{equation}		
\end{lemma}
\begin{proof}
Using a local orthonormal basis $\{e_i\}_{i=1,n}$  parallel at the point where the computation is done, we obtain for any vector field $X\in \Gamma(\T M)$ assumed to be parallel at the given point:
	\begin{equation*}
		\begin{split}
\mathrm{Ric}(X,\xi)&=\sum_{i=1}^n \langle R_{e_i, X}\xi, e_i\rangle=\sum_{i=1}^n \left(\langle \nabla_{e_i}\nabla_X\xi, e_i\rangle-\langle\nabla_X\nabla_{e_i}\xi, e_i\rangle\right)\\
			&\overset{\eqref{koxi}}{=}\sum_{i=1}^n \left(\langle \nabla_{e_i}(-\theta(\xi)X+\langle X, \xi\rangle\theta), e_i\rangle-X(\langle\nabla_{e_i}\xi, e_i\rangle)\right)\\
			&=-X(\theta(\xi))+\langle X,\nabla_\theta \xi\rangle-\langle X, \xi\rangle \delta\theta+X(\delta\xi)\\
			&\overset{\eqref{codiffxi}}{=}-X(\theta(\xi))-\langle X, \xi\rangle \delta\theta+(n-1)X(\theta(\xi))\\
			&=(n-2)X(\theta(\xi))-\langle X, \xi\rangle \delta\theta,
		\end{split}
	\end{equation*}		
which proves \eqref{ricxi}. For the expression of $\Ric(\J\xi)$ we compute similarly:
\begin{equation*}
	\begin{split}
		\mathrm{Ric}(X,\J\xi)&=\sum_{i=1}^n \langle R_{e_i, X}\J\xi, e_i\rangle=\sum_{i=1}^n \left(\langle \nabla_{e_i}\nabla_X\J\xi, e_i\rangle-\langle\nabla_X\nabla_{e_i}\J\xi, e_i\rangle\right)\\
		&\overset{\eqref{koxi}}{=}\sum_{i=1}^n \left(\langle \nabla_{e_i}(-\theta(\xi)\J X+\langle X, \xi\rangle\J\theta), e_i\rangle-X(\langle\nabla_{e_i}\J\xi, e_i\rangle)\right)\\
		&=-\J X(\theta(\xi))+\langle X,\nabla_{\J \theta}\xi\rangle-\langle X, \xi\rangle \delta(\J \theta)+X(\delta\J\xi)\\
		&\overset{\eqref{koxi}}{=}-\J(\di(\theta(\xi)))(X)-\theta(\xi)\<X, \J\theta\>+\theta( \J\xi)\<X, \theta\>-\langle X, \xi\rangle \delta(\J \theta)+(\di(\delta\J\xi))(X).
	\end{split}
\end{equation*}		
This formula, together with
\begin{equation*}
	\begin{split}
		\delta(\J\xi)=-\sum_{i=1}^n e_i\lrcorner(\J\nabla_{e_i}\xi)\overset{\eqref{koxi}}{=}-\sum_{i=1}^n e_i\lrcorner(-\theta(\xi)\J\e_i+\<e_i, \xi\>\J\theta)=\tr(\J)\theta(\xi)-\theta(\J\xi).
	\end{split}
\end{equation*}		
yields \eqref{ricJxi}.
\end{proof}	

Since $\Ric$ and $\J$ commute,  we have in particular  $\Ric(\J\xi)=\J \Ric(\xi)$, which according to \eqref{ricxi} and \eqref{ricJxi} reads:
\begin{equation}\label{eqric}
	\theta(\J \xi)\theta-\theta(\xi)\J\theta-(\delta\J\theta)\xi+(\delta\theta)\J\xi-\di(\theta(\J\xi))-(n-1)\J\di(\theta(\xi))+\tr(\J)\di(\theta(\xi))=0.
\end{equation}		
  We now establish some further results under the hypotheses of Theorem~\ref{thmcond*}.

\begin{proposition}\label{proptheta-}
If  condition $(*)$ is satisfied, then $\xi$ and $\J\xi$ are contained in the kernel of the Lee form and have constant scalar product: 
\begin{equation}\label{eqtheta-}
	\theta(\xi)=0,
\end{equation}	
\begin{equation}\label{eqtheta-J}
	\theta(\J\xi)=0,
\end{equation}	
\begin{equation}\label{constangl}
\di(\<\xi, \J \xi\>)=0.
\end{equation}	
\end{proposition}	

\begin{proof}
According to Theorem~\ref{thma+a-}, the assumption $\rk(E_-(S))=1$ implies that $A_+=0$, so Equality \eqref{integral} reads:
\begin{equation}\label{integralhalf}
	\int_M A_-\|\theta_-\|^2\vol^g=0.
\end{equation}
Since $A_-$ is non-negative, the above equality implies that at any point $x\in M$ either $\theta_-(x)=0$ or $A_-(x)=0$. We consider the following closed set of~$M$:
$$C:=\{x\in M\, |\, \theta_-(x)=0\}=\{x\in M\, |\, \theta(\xi)(x)=0\}.$$
We claim that $C=M$. Assume for a contradiction that there exists $x_0\in M\setminus C$. Then there exists a connected neighbourhood $V_0\subseteq M\setminus C$ of $x_0$ on which the function $\|\theta_-\|^2$ is strictly positive. By \eqref{integralhalf} we obtain that the non-negative function $A_-$ must vanish on $V_0$. Furthermore, by the equivalence \eqref{equiva-}, $A_-(x)=0$ means that 
there exists $\varepsilon\in\{-1,1\}$, such that  $E_{+}(S)\subseteq E_{\varepsilon}(\J)$ at $x$. Again without loss of generality, we may assume that $\varepsilon=1$ at each point of the connected neighbourhood $V_0$, so $E_{+}(S)\subseteq E_{+}(\J)$ at each $x\in V_0$. Taking the orthogonal complements yields  $E_{-}(\J)\subseteq E_{-}(S)$, which implies $E_{-}(\J)=E_{-}(S)$, because $\rk(E_-(S))=1$ and $\J$ is not the identity of $\T M$. Thus, we obtain that $\J=S$ on $V_0$. In particular, since $\J$ is $\nabla$-parallel, the vector field $\xi$ must also be parallel on $V_0$: $\nabla_X\xi=0$, for all vector fields $X$. Then \eqref{koxi} implies that on $V_0$ the Lee form $\theta$ vanishes, so in particular $\theta_-=0$, but this contradicts the fact that $V_0$ is contained in $M\setminus C$. Hence, our assumption is false, which yields $M=C$ and proves \eqref{eqtheta-}.

Substituting \eqref{eqtheta-} into Equality~\eqref{eqric} yields:
\begin{equation}\label{eqric1}
	\theta(\J \xi)\theta-\di(\theta(\J\xi))-(\delta\J\theta)\xi+(\delta\theta)\J\xi=0.
\end{equation}		
On the other hand, applying the exterior derivative to Equality ~\eqref{eqxiJxi} yields:
$$0=\di^2(\<\xi, \J\xi\>)=2\di(\theta(\J\xi))\wedge \xi+2\theta(\J\xi)\di\xi\overset{\eqref{codiffxi}}{=}2\left(\di(\theta(\J\xi))-\theta(\J\xi)\theta\right)\wedge\xi,$$
so taking the exterior product with $\xi$ in \eqref{eqric1}, we obtain:
\begin{equation*}
	(\delta\theta)\J\xi\wedge \xi=0.
\end{equation*}		
Taking the interior product with $\theta$ in this formula and using \eqref{eqtheta-} shows that $(\delta\theta)\theta(\J\xi)=0$. Substituting now this term in the equality obtained by taking in \eqref{eqric1} the scalar product with $\theta$ yields $$\|\theta\|^2\theta(\J\xi)=\<\theta,\di(\theta(\J\xi))\>.$$
 At any extremum point of the function $\theta(\J\xi)$, the right-hand side of this equality vanishes. In particular, this equality shows that at a maximum and minimum point of $\theta(\J\xi)$, which exist since the manifold is compact, the function $\theta(\J\xi)$ vanishes (because at such a point $x_0$ either $\theta(\J\xi)(x_0)=0$ or $\theta(x_0)=0$, which again implies in particular that $\theta(\J\xi)(x_0)=0$). Hence $\theta(\J\xi)$ is identically zero on $M$.
Finally, \eqref{constangl} follows directly from \eqref{eqxiJxi}, \eqref{eqtheta-} and \eqref{eqtheta-J}.
\end{proof}
 
 \begin{corollary}
 	If  condition $(*)$ is satisfied, then the covariant derivative of $\xi$ and the Riemannian curvature tensor of $g$ applied to $\xi$ are given by the following formulas, for all vector fields $X,Y$: 
 	\begin{equation}\label{koxi1}
 	\nabla_X\xi=\langle X, \xi\rangle\theta,
 	\end{equation}	
 		\begin{equation}\label{riemxi}
 			\begin{split}
 				R_{X,Y} \xi=\langle Y,\theta\rangle\langle X, \xi\rangle\theta+\langle Y, \xi\rangle\nabla_X\theta-\langle X,\theta\rangle\langle Y, \xi\rangle\theta-\langle X, \xi\rangle\nabla_Y\theta.
 			\end{split}	
 		\end{equation}
 \end{corollary}	
 \begin{proof}
 The formula \eqref{koxi1} follows directly from \eqref{koxi} and \eqref{eqtheta-}. In order to show \eqref{riemxi}, we compute for any vector fields $X,Y$ parallel at the point where the computation is done:
 \begin{equation*}
 	\begin{split}
 		R_{X,Y} \xi&=\nabla_X\nabla_Y\xi-\nabla_Y\nabla_X\xi\overset{\eqref{koxi1}}{=}\nabla_X(\langle Y, \xi\rangle\theta)-\nabla_Y(\langle X, \xi\rangle\theta)\\
 		&=\langle Y, \nabla_X \xi\rangle\theta+\langle Y, \xi\rangle\nabla_X\theta-\langle X,  \nabla_Y\xi\rangle\theta-\langle X, \xi\rangle\nabla_Y\theta\\
 		&\overset{\eqref{koxi1}}{=}\langle Y,\theta\rangle\langle X, \xi\rangle\theta+\langle Y, \xi\rangle\nabla_X\theta-\langle X,\theta\rangle\langle Y, \xi\rangle\theta-\langle X, \xi\rangle\nabla_Y\theta.
 	\end{split}	
 \end{equation*}
\end{proof}

We consider the decomposition of the tangent bundle given by the fact that the holonomy of $g$ is reducible: $\T M=E_+(\J)\oplus E_-(\J)$ and we denote by $\theta=\theta_1+\theta_2$  the corresponding splitting of the Lee form, where 
$$\theta_1:=\frac{1}{2}\left( \theta+\J \theta\right) \quad { and }\quad  \theta_2:=\frac{1}{2}\left( \theta-\J \theta\right) .$$
Also the vector field $\xi$ decomposes as $\xi=\xi_1+\xi_2$ with $\xi_1:=\frac{1}{2}\left( \xi+\J \xi\right)$ and $\xi_2:=\frac{1}{2}\left( \xi-\J \xi\right)$. Let us notice that both components have constant lengths due to \eqref{constangl}: $\|\xi_1\|^2=\frac{1}{2}(1+\<\xi, \J\xi\>)$ and  $\|\xi_2\|^2=\frac{1}{2}(1-\<\xi, \J\xi\>)$ are constant functions. In particular, each of them either vanishes everywhere on $M$ or has no zeros. 

{\bf Claim.} One of the vector fields  $\xi_1$ or $\xi_2$ vanishes everywhere on $M$. 

{\sl Proof of the Claim.} Assume for a contradiction  that both $\xi_1$ and $\xi_2$ are not identically zero. By the above remark, $\xi_1$ and $\xi_2$ have no zeros. In this case applying equality~\eqref{riemxi} to  some vector field $X=X_1\in E_+(\J)$ and to $Y=\xi_2\in E_-(\J)$ and using the fact that $\theta(\xi_2)=0$ (which is a consequence of \eqref{eqtheta-} and \eqref{eqtheta-J}), we obtain:
\begin{equation}\label{riemxi1}
	\begin{split}
		0=\| \xi_2\|^2\nabla_{X_1}\theta-\theta_1(X_1)\| \xi_2\|^2\theta-\langle X_1, \xi_1\rangle\nabla_{\xi_2}\theta,
	\end{split}	
\end{equation}
yielding
\begin{equation}\label{riemxi2}
	\begin{split}
		\nabla_{X_1}\theta=\theta_1(X_1)\theta+\frac{\langle X_1, \xi_1\rangle}{\| \xi_2\|^2}\nabla_{\xi_2}\theta.
	\end{split}	
\end{equation}


Choosing now $X_1=\theta_1$ and taking the scalar product with $\theta_1$ yields $\theta_1(\|\theta_1\|^2)=2\|\theta_1\|^4$. In particular, this equality shows that $\|\theta_1\|^2$ vanishes at its maximum, so $\theta_1$ vanishes on $M$. The same argument (exchanging the indices 1 and 2) shows that $\theta_2$ vanishes on $M$. Thus $\theta=0$, in contradiction to the hypothesis of Theorem~\ref{thmcond*}. Thus our assumption was false, which proves the Claim. \qed

By replacing $\J$ with $-\J$ if necessary, we can thus assume that $\xi_2$ vanishes on $M$. This means that $\xi=\xi_1\in E_{+}(\J)$, so in this case $E_{-}(S)\subseteq E_{+}(\J)$, showing $(ii)$ of Theorem~\ref{thmcond*}. In particular it follows that the tangent bundle splits as an orthogonal  direct sum $\T M=E_{-}(S)\oplus K\oplus E_{-}(\J)$, where $K$ is the orthogonal complement of $E_{-}(S)$ in $E_{+}(\J)$. On each of these three subspaces the endomorphisms $S$ and $\J$ act either as the identity or minus the identity, thus $S$ and $\J$ commute, proving $(i)$ of Theorem~\ref{thmcond*}.

\section{Appendix}

In this appendix we express all traces occurring in Equality~\eqref{traceidentity0} in terms of codifferentials:

\begin{lemma}\label{traces}
	The following identities hold:
	\begin{equation}\label{trT}
		\tr(T)=-\delta\theta,
    \end{equation}	
\begin{equation}\label{trJT}
	\tr(\J T)=-\delta(\J \theta),
\end{equation}	
\begin{equation}\label{trST}
	\tr(ST)=-\delta(S\theta)-\|\theta\|^2\tr(S)+n\<\theta, S\theta\>,
\end{equation}	
\begin{equation}\label{trJST}
	\tr(\J ST)=-\delta(\J S\theta)-\|\theta\|^2\tr(S\J)+\<\theta, S\theta\>\tr(\J),
\end{equation}	
\begin{equation}\label{trSJT}
	\tr(S\J T)=-\delta(S \J \theta)-\<\theta, \J\theta\>\tr(S)+n\<\theta, S\J \theta\>,
\end{equation}	
\begin{equation}\label{trJSJT}
	\tr(\J S\J T)=-\delta(\J S \J \theta)-\<\theta, \J S \J\theta\>-\<\theta,\J\theta\>\tr(S\J)+\<\theta, S\J \theta\>\tr(\J)+\<\theta, S\theta\>,
\end{equation}	
\begin{equation}\label{trSJST}
	\begin{split}
		\tr(S\J  ST)=&-\delta(S \J S \theta)+(n+1)\<\theta, S\J S\theta\>-\<\theta, S\J\theta\>\tr(S)\\
		&-\<\theta, \J \theta\>-\|\theta\|^2\tr(\J)+\<\theta, S\theta\>\tr(S\J),
	\end{split}
\end{equation}	
		\begin{equation}\label{trSJSJT}
			\begin{split}
				\tr(S\J S\J T)=&-\delta(S\J S\J \theta)+(n+1)\<\theta, S\J S\J \theta\>-\<\theta, \J S\J\theta\>\tr(S)\\
				&-\|\theta\|^2-\<\theta, \J\theta\>\tr (\J)+\<\theta, S\J \theta\>\tr(S\J),
			\end{split}
		\end{equation}
		\begin{equation}\label{trJSJST}
			\begin{split}
				\tr(\J S\J ST)=&-\delta(\J S\J S\theta)-\<\theta, S\J S\J\theta\>-\<\theta, S\J \theta\>\tr(S\J)+\<\theta, S\J S\theta\>\tr (\J)\\
				&+\|\theta\|^2(1-\tr(S\J S\J))+\<\theta, S\theta\>\tr(S).
			\end{split}
		\end{equation}
\end{lemma}

\begin{proof}
	Equality \eqref{trT} follows directly from the definition of $T:=\nabla\theta$.  Equality \eqref{trJT} follows from \eqref{trT} and the fact that $\J$ is $\nabla$-parallel. For the remaining formulas, we use \eqref{derivS}.
If $\{e_i\}_{i=1,n}$  is a local orthonormal basis of $\T M$, then we obtain:
		\begin{equation*}
			\begin{split}
				\tr(ST)&= \sum_{i=1}^n \<ST e_i, e_i\>=\sum_{i=1}^n \<S\nabla_{e_i}\theta, e_i\>=\sum_{i=1}^n \<\nabla_{e_i}(S\theta), e_i\>-\sum_{i=1}^n \<(\nabla_{e_i}S) \theta, e_i\>\\
				&\overset{\eqref{derivS}}{=}-\delta(S\theta)-\sum_{i=1}^n\<(Se_i\odot\theta -S\theta\odot e_i)(\theta), e_i\>\\
				&=-\delta(S\theta)-\sum_{i=1}^n\left(\<Se_i, \theta\>\<\theta,e_i\>+\|\theta\|^2\<Se_i, e_i\>-\<S\theta, \theta\>\<e_i, e_i\>-\<e_i,\theta\>\<S\theta, e_i\>\right)\\
				&=-\delta(S\theta)-\tr(S)\|\theta\|^2+n\<\theta, S\theta\>.
			\end{split}
		\end{equation*}
We compute similarly the other traces:
		\begin{equation*}
			\begin{split}
				\tr(\J ST)&= \sum_{i=1}^n \<\J ST e_i, e_i\>=\sum_{i=1}^n \<\J S\nabla_{e_i}\theta, e_i\>=\sum_{i=1}^n \<\nabla_{e_i}(\J S\theta), e_i\>-\sum_{i=1}^n \<\J (\nabla_{e_i}S) \theta, e_i\>\\
				&\overset{\eqref{derivS}}{=}-\delta(\J S\theta)-\sum_{i=1}^n\<(Se_i\odot\theta -S\theta\odot e_i)(\theta), \J e_i\>\\
				&=-\delta(\J S\theta)-\sum_{i=1}^n\<\<Se_i, \theta\>\theta+\|\theta\|^2Se_i-\<S\theta, \theta\>e_i-\<e_i,\theta\>S\theta, \J e_i\>\\
				&=-\delta(\J S\theta)-\<S\theta, \J \theta\>-\|\theta\|^2\tr(S\J )+\<\theta, S\theta\>\tr(\J)+\<\theta, \J S\theta\>\\
				&=-\delta(\J S\theta)-\|\theta\|^2\tr(S\J )+\<\theta, S\theta\>\tr(\J).
			\end{split}
		\end{equation*}
		\begin{equation*}
			\begin{split}
				\tr(S\J T)&= \sum_{i=1}^n \<S\J T e_i, e_i\>=\sum_{i=1}^n \<S\J \nabla_{e_i}\theta, e_i\>=\sum_{i=1}^n \<\nabla_{e_i}(S \J \theta), e_i\>-\sum_{i=1}^n \< (\nabla_{e_i}S) \J\theta, e_i\>\\
				&\overset{\eqref{derivS}}{=}-\delta(S \J \theta)-\sum_{i=1}^n\<(Se_i\odot\theta -S\theta\odot e_i)(\J \theta), e_i\>\\
				&=-\delta(S\J \theta)-\sum_{i=1}^n\<\<Se_i, \J \theta\>\theta+\<\theta, \J\theta\>Se_i-\<S\theta, \J \theta\>e_i-\<e_i,\J \theta\>S\theta, e_i\>\\
				&=-\delta(S \J\theta)-\<\theta,  S\J \theta\>-\<\theta, \J \theta\>\tr(S)+n\<S\theta, \J \theta\>+\<S \theta, \J\theta\>\\
				&=-\delta(S \J \theta)-\<\theta, \J \theta\>\tr(S)+n\<\theta, S\J \theta\>.
			\end{split}
		\end{equation*}
	\begin{equation*}
		\begin{split}
			\tr(\J S\J T)=& \sum_{i=1}^n \<\J S\J T e_i, e_i\>=\sum_{i=1}^n \<\J S\J \nabla_{e_i}\theta, e_i\>=\sum_{i=1}^n \<\nabla_{e_i}(\J S \J \theta)- \J (\nabla_{e_i}S) \J\theta, e_i\>\\
			\overset{\eqref{derivS}}{=}&-\delta(\J S \J \theta)-\sum_{i=1}^n\<(Se_i\odot\theta -S\theta\odot e_i)(\J \theta), \J e_i\>\\
			=&-\delta(\J S\J \theta)-\sum_{i=1}^n\left(\<Se_i, \J \theta\>\<\theta, \J e_i\>+\<\theta, \J\theta\>\<Se_i, \J e_i\>\right)\\
			&+\sum_{i=1}^n\left(\<S\theta, \J \theta\>\<e_i,  \J e_i\>+\<e_i,\J \theta\>\<S\theta, \J e_i\>\right)\\
			=&-\delta(\J S \J\theta)-\<\theta,  \J S\J \theta\>-\<\theta, \J \theta\>\tr(S\J)+\<\theta, S \J \theta\>\tr(\J)+\<\theta, S \theta\>.
		\end{split}
	\end{equation*}
\begin{equation*}
	\begin{split}
		&\tr(S \J S T)=\sum_{i=1}^n \<S \J S T e_i, e_i\>=\sum_{i=1}^n \<S \J S \nabla_{e_i}\theta, e_i\>=\\
		=&\sum_{i=1}^n \<\nabla_{e_i}(S \J S \theta), e_i\>-\sum_{i=1}^n \< (\nabla_{e_i}S)\J S\theta, e_i\>-\sum_{i=1}^n \< S \J (\nabla_{e_i}S)\theta, e_i\>\\
		\overset{\eqref{derivS}}{=}&-\delta(S \J S \theta)-\sum_{i=1}^n\<(Se_i\odot\theta -S\theta\odot e_i)(\J S \theta), e_i\>-\sum_{i=1}^n\<(Se_i\odot\theta -S\theta\odot e_i)(\theta), \J S e_i\>\\
		=&-\delta(S \J S \theta)-\sum_{i=1}^n\<\<Se_i, \J S\theta\>\theta+\<\theta, \J S\theta\>Se_i+\<S\theta, \J S\theta\>e_i+\<e_i,\J S\theta\>S\theta, e_i\>\\
		&-\sum_{i=1}^n\<\<Se_i, \theta\>\theta+\<\theta, \theta\>Se_i
		+\<S\theta, \theta\>e_i+\<e_i,\theta\>S\theta, \J S e_i\>\\
		=&-\delta(S \J S\theta)-\<\theta,  S\J S \theta\>-\<\theta,  \J S\theta\>\tr(S)+n\<\theta,  S\J S \theta\>+\<\theta,  S\J S \theta\>\\
		&-\<\theta,  \J\theta\>-\|\theta\|^2\tr(\J)+\<\theta, S \theta\>\tr(S\J)+\<\theta,  S \J S\theta\>\\
		=&-\delta(S \J S\theta)+(n+1)\<\theta,  S\J S \theta\>-\<\theta,  S\J \theta\>\tr(S)-\<\theta,  \J\theta\>-\|\theta\|^2\tr(\J)+\< \theta, S\theta\>\tr(S\J).
	\end{split}
\end{equation*}

\begin{equation*}
	\begin{split}
		&\tr(S \J S \J T)=\sum_{i=1}^n \<S \J S \J T e_i, e_i\>=\sum_{i=1}^n \<S \J S \J\nabla_{e_i}\theta, e_i\>=\\
		=&\sum_{i=1}^n \<\nabla_{e_i}(S \J S \J \theta), e_i\>-\sum_{i=1}^n \< (\nabla_{e_i}S)\J S\J \theta, e_i\>-\sum_{i=1}^n \< S \J (\nabla_{e_i}S)\J\theta, e_i\>\\
		\overset{\eqref{derivS}}{=}&-\delta(S \J S \J \theta)-\sum_{i=1}^n\<(Se_i\odot\theta -S\theta\odot e_i)(\J S \J\theta), e_i\>\\&-\sum_{i=1}^n\<(Se_i\odot\theta -S\theta\odot e_i)(\J \theta), \J S e_i\>\\
		=&-\delta(S \J S \J\theta)-\sum_{i=1}^n\<\<Se_i, \J S\J \theta\>\theta+\<\theta, \J S\J\theta\>Se_i+\<S\theta, \J S\J\theta\>e_i+\<e_i,\J S\J\theta\>S\theta, e_i\>\\
		&-\sum_{i=1}^n\<\<Se_i, \J\theta\>\theta+\<\theta, \J\theta\>Se_i+\<S\theta, \J \theta\>e_i+\<e_i,\J\theta\>S\theta, \J S e_i\>\\
		=&-\delta(S \J S\J \theta)-\<\theta,  S\J S \J\theta\>-\<\theta,  \J S\J \theta\>\tr(S)+n\<\theta,  S\J S \J\theta\>+\<\theta,  S\J S \J\theta\>-\|\theta\|^2\\
		&-\<\theta,  \J\theta\>\tr(P)+\<\theta,  S\J\theta\>\tr(S\J)+\<\theta,  S\J S \J\theta\>\\
		=&-\delta(S \J S\J \theta)+(n+1)\<\theta,  S\J S \J\theta\>-\<\theta,  \J S\J \theta\>\tr(S)-\|\theta\|^2\\&-\<\theta,  \J\theta\>\tr(P)+\<\theta,  S\J\theta\>\tr(S\J).
	\end{split}
\end{equation*}

Finally, we compute :

\begin{equation*}
	\begin{split}
		&\tr(\J S \J S T)=\sum_{i=1}^n \<\J S \J S T e_i, e_i\>=\sum_{i=1}^n \<\J S \J S \nabla_{e_i}\theta, e_i\>=\\
		=&\sum_{i=1}^n \<\nabla_{e_i}(\J S \J S \theta), e_i\>-\sum_{i=1}^n \< \J (\nabla_{e_i}S)\J S \theta, e_i\>-\sum_{i=1}^n \<\J S \J (\nabla_{e_i}S)\theta, e_i\>\\
		\overset{\eqref{derivS}}{=}&
		-\delta(\J S \J S\theta)-\sum_{i=1}^n\<\<Se_i, \J S \theta\>\theta+\<\theta, \J S\theta\>Se_i+\<S\theta, \J S\theta\>e_i+\<e_i,\J S\theta\>S\theta, \J e_i\>\\
		&-\sum_{i=1}^n\<\<Se_i, \theta\>\theta+\<\theta, \theta\>Se_i+\<S\theta, \theta\>e_i+\<e_i,\theta\>S\theta, \J S\J e_i\>\\
		=&-\delta(\J S \J S\theta)-\<\theta,  S\J S \J\theta\>-\<\theta,  S\J \theta\>\tr(S\J)+\<\theta, S\J S\theta\>\tr(\J)\\
		&+\|\theta\|^2(1-\tr(S\J S\J))+\<\theta, S\theta\>\tr(S).
	\end{split}
\end{equation*}

\end{proof}


\begin{thebibliography}{5}

\bibitem{besse} A.L. Besse, \emph{Einstein manifolds}. Classics in Mathematics, Springer-Verlag, Berlin, 2008.

\bibitem{bfm2023} F. Belgun, B. Flamencourt, A. Moroianu, {\sl Weyl structures with special holonomy on compact conformal manifolds}. arxiv2305.06637 (2023).

\bibitem{bm2011} F. Belgun, A. Moroianu, {\sl Weyl-parallel forms, conformal products and Einstein-Weyl manifolds}. Asian J. Math. {\bf 15}, 499--520 (2011).

\bibitem{bm2016} F. Belgun, A. Moroianu, {\sl On the irreducibility of locally metric connections.} J. reine angew. Math. {\bf 714}, 123--150 (2016).

\bibitem{c2000} D.M.J. Calderbank, {\sl Selfdual Einstein metrics and conformal submersions}. arXiv:math/0001041v1 (2000).

\bibitem{f2024}  B. Flamencourt, {\sl Locally conformally product structures.} Internat. J. Math. {\bf 35} (5), 2450013, (2024).

\bibitem{fz2025} B. Flamencourt, A. Zeghib, {\sl On foliations admitting a transverse similarity structure}. arxiv2501.04814 (2025).

\bibitem{Frie} D. Fried, {\sl Closed similarity manifolds}. Comment. Math. Helv. {\bf 55} (4), 576--582 (1980).

\bibitem{g1995} P. Gauduchon, {\sl Structures de Weyl-Einstein, espaces de twisteurs et vari\'et\'es de type $S^1\times S^3$}.
J. reine angew. Math. {\bf 469}, 1--50 (1995).

\bibitem{k2019} M. Kourganoff, {\sl Similarity structures and de Rham decomposition}. Math. Ann. {\bf 373}, 1075--1101  (2019).

\bibitem{k1988} W. Kühnel, {\sl Conformal transformations between Einstein spaces}. Conformal Geometry. Aspects of Math. {\bf 12} (R.S. Kulkarni, U. Pinkall, eds.), 105--146 (1988).

\bibitem{kr1997} W. K\"uhnel, H.-B. Rademacher, {\sl Conformal vector fields on pseudo-Riemannian spaces}. Diff. Geom. Appl. {\bf 7}, 237--250 (1997).

\bibitem{kr2009} W. K\"uhnel, H.-B. Rademacher, {\sl Einstein spaces with a conformal group}. Res. Math. {\bf 56}, 421--444 (2009).

\bibitem{kr2016} W. K\"uhnel, H.-B. Rademacher, {\sl Conformally Einstein product spaces}, Differential Geom. Appl. {\bf 49}, 65--96 (2016).

\bibitem{ms1999} S. Merkulov, L. Schwachhöfer, {\sl Classification of irreducible holonomies of torsion-free affine connections.} Annals of Math. {\bf 150} (1), 77--149 (1999).

\bibitem{mn2015} V. Matveev, Y. Nikolayevsky, {\sl A counterexample to Belgun-Moroianu conjecture.} C. R. Math. Acad. Sci. Paris {\bf 353}, 455--457 (2015).

\bibitem{mn2017} V. Matveev, Y. Nikolayevsky, {\sl Locally conformally Berwald manifolds and compact quotients of reducible manifolds by homotheties.} Ann. Inst. Fourier (Grenoble) {\bf 67} (2), 843--862 (2017).

\bibitem{m2019} A. Moroianu, {\sl Conformally related Riemannian metrics with non-generic holonomy}. J. Reine Angew. Math. {\bf  755}, 279--292 (2019).

\bibitem{mp2024b} A. Moroianu, M. Pilca, {\sl Einstein metrics on conformal products}. Ann. Global Anal. Geom. {\bf 65}, 20 (2024). 

\bibitem{mp2025} A. Moroianu, M. Pilca, {\sl Conformal product structures on compact K\"ahler manifolds}, Adv. Math. {\bf 467}, article 110181 (2025).

\bibitem{mp2025a}  A. Moroianu, M. Pilca, {\sl Conformal product structures on compact Einstein manifolds}, arxiv2504.07886 (2025).

\bibitem{v1979} I. Vaisman, {\sl Conformal foliations}, Kodai Math. J. {\bf 2}, 26--37 (1979).

\end{thebibliography}
\end{document}